\documentclass[letterpaper]{article}
\usepackage{geometry}

\usepackage{amsmath,mathtools}

\usepackage[utf8x]{inputenc}
\usepackage[T1]{fontenc}
\usepackage{graphicx}
\usepackage{subcaption}
\usepackage{longtable} 
\usepackage{multirow}
\usepackage{listings}
\usepackage{makecell}
\usepackage{longtable}
\usepackage{float}
\usepackage[dvipsnames]{xcolor}
\usepackage{graphicx}
\usepackage{natbib}
\usepackage{amsmath} 
\usepackage{amssymb}  
\usepackage{amsfonts}
\usepackage{amsthm}
\usepackage{setspace}
\usepackage{mathtools}
\usepackage{mathrsfs}
\usepackage{pgf}
\usepackage{dsfont}
\usepackage{xspace}
\usepackage[hyperindex,breaklinks]{hyperref}
\hypersetup{
    anchorcolor=red,
    colorlinks,%
    citecolor=blue,%
    filecolor=blue,%
    linkcolor=blue,%
    urlcolor=blue
}
\PassOptionsToPackage{linktocpage}{hyperref}
\usepackage[toc,page,titletoc]{appendix}

\DeclareMathOperator{\ex}{\mathbb{E}}

\DeclareMathOperator{\prob}{\mathbb{P}}

\DeclareMathOperator{\tr}{\mathsf{Tr}}

\DeclarePairedDelimiterXPP{\dnorm}[1]{}{\lVert}{\rVert}{_{2}}{#1}

\DeclarePairedDelimiterX{\braket}[2]{\langle}{\rangle}{#1\mathopen{}\delimsize,\mathopen{}#2}

\DeclarePairedDelimiterX{\inner}[2]{\langle}{\rangle}{#1,#2}
\DeclarePairedDelimiterX{\setdef}[2]{\{}{\}}{#1:#2}

\DeclarePairedDelimiterXPP{\exclude}[1]{\mathopen{}\setminus}{\{}{\}}{}{#1}

\DeclarePairedDelimiterXPP{\probof}[1]{\prob}{(}{)}{}{%
	
	#1}

\DeclarePairedDelimiterXPP{\exof}[1]{\ex}{[}{]}{}{%
	
	#1}

\DeclarePairedDelimiterXPP{\trof}[1]{\tr}{[}{]}{}{#1}

\newtheorem{lem}{Lemma}
\newtheorem{thm}{Theorem}
\newtheorem{rem}{Remark}
\newtheorem{defn}{Definition}

\usepackage{algorithm}
\usepackage{algorithmic}

\numberwithin{equation}{section}
\numberwithin{theorem}{section}
\numberwithin{remark}{section}
\numberwithin{example}{section}

\vspace{-10ex}
\title{Comments on Efficient Singular \\Value Thresholding Computation}

\author{
Zhengyuan Zhou\thanks{Stern School of Business, New York University. Email: zzhou@stern.nyu.edu.}
\and Yi Ma\thanks{UC Berkeley EECS. Email: yima@eecs.berkeley.edu}
}
\date{}

\begin{document}
\maketitle
\begin{abstract}
We discuss how to evaluate the proximal operator of a convex and increasing function of a nuclear norm, which forms the key computational step in several first-order optimization algorithms such as (accelerated) proximal gradient descent and ADMM.
Various special cases of the problem arise in low-rank matrix completion, dropout training in deep learning and high-order low-rank tensor recovery, although they have all been solved on a case-by-case basis. We provide an unified and efficiently computable procedure for solving this problem. 
\end{abstract}

\section{Problem, Notation and Background}
Proximal gradient descent (and its accelerated variant) (\cite{beck2009gradient, combettes2011proximal, ma2012alternating, parikh2014proximal}) provide an efficient way (with $O(\frac{1}{T})$ and $O(\frac{1}{T^2})$ convergence rates, respectively) to solve structured 
convex non-smooth optimizaiton problems of the following form, which arise frequently in machine learning and structured signal recovery settings:
$$\min_{x} F(x) = g(x) + h(x),$$
where $g(x)$ is convex and smooth (i.e. continuously differentiable with bounded gradients) and $h(x)$ is convex but non-smooth. The presence of $h$ often results from non-smooth but convex regularizers such as $l_1$-norm or nuclear norm, as two prominent examples.
The computational bottleneck in each iteration of (accelerated) proximal gradient is evaluating the proximal operator (at a point $w$):
$$\text{prox}_{h}[w] = \arg\min_{x} \{h(x) + \frac{1}{2} \|x-w\|_2^2 \},$$
which admits a unique solution due to strong convexity. Computationally, efficient (accelerated) proximal gradient is possible when and only when $\text{prox}_{h}[w] $ can be computed efficiently. 
Note that in addition to (accelerated) proximal gradient, evaluating proximal operators is also the key and computationally demanding step in methods such as ADMM (\cite{luo2012linear, yang2013linearized,journals/ftml/BoydPCPE11}) for large-scale equality constrained optimization problems.

In this note, we focus on a class of non-smooth convex functions $h: \mathbf{R}^{n_1\times n_2} \rightarrow \mathbf{R}$, where $h(X) = f(\|X\|_*)$ for some convex and increasing $f: \mathbf{R} \rightarrow \mathbf{R}$, with $\|X\|_*$ being the nuclear norm of $X$. Thus, our main question is, how to efficient compute the optimal solution $X^*$, where:
\begin{equation}\label{m1}
X^* = \arg\min \{ \tau f(\|X\|_*) + \frac{1}{2} \|X - Y\|_F^2\}.
\end{equation}

It turns out a few special cases of this problem have played important roles in machine learning and signal processing. For instance, when $f(x) = x$, this problem occurs in low-rank matrix completion and recovery problems (\cite{cai2010singular}); when $f(x) = x^2$, this problem occurs in drop-out training in deep learning (\cite{cavazza2018dropout}); when $f(x) = e^x$, this problem occurs in high-order low-rank tensor recovery (\cite{zhang2014hybrid}). Further, analytical and/or efficiently computable solutions have been derived in these special cases. For instance, consider a matrix $Y\in\mathbf{R}^{n_1 \times n_2}$ of rank $r$, whose singular value decomposition (SVD) is:
$$Y = U \Sigma V^*, \Sigma = \mathrm{diag}(\{\sigma^i_Y\}_{1\le i \le r}), \sigma^1_Y \ge \sigma^2_Y \ge \dots \ge \sigma^r_Y > 0,$$ where $U$ and $V$ are $n_1 \times r$
and $n_2 \times r$ matrices respectively with orthonormal columns. 
For each $\tau > 0$, the soft thresholding shrinkage operator $\mathcal{D}_\tau$ is defined to be:
$$\mathcal{D}_\tau(Y) = U\mathcal{D}_\tau(\Sigma)V^*, \ \mathcal{D}_\tau(\Sigma) = \mathrm{diag}({(\sigma^i_Y - \tau)_+}),t_+ = \max(0,t).$$
\cite{cai2010singular} shows that soft thresholding provides an analytical solution when $f(x) = x$.
\begin{lem}\label{candes} [\cite{cai2010singular}]
	Given a matrix $Y\in\mathbf{R}^{n_1 \times n_2}$ and a $\tau \ge 0$, consider the function $h(X) = \tau \|X\|_* + \frac{1}{2}\|X-Y\|^2_F, \ X\in\mathbf{R}^{n_1 \times n_2}.$
	We have $\mathrm{arg}\min_{X}{h(X)} = \mathcal{D}_\tau(Y)$.
\end{lem}

However, although the proximal mappings of these different cases have been solved separately,
there is a lack of an unified and efficiently computable scheme to solve the above problem for a general $f$ (one should not expect an analytical solution exists for a general $f$).
It turns out that soft singular value thresholding still works and the threshold can be computed efficiently by a binary search on a system of 1-dimensional equations that depend on the input data matrix $Y$ and $f(\cdot)$.

\section{Main Results}

We augment the list of non-increasing singular values $\{\sigma^i_Y\}_{1\le i \le r}$ of $Y$ with $\sigma^{r+1}_Y = -\infty$.
\begin{lem}\label{existence}
	Let $g: \mathbf{R} \rightarrow \mathbf{R}$ be such that $g(x) \ge 0$,     $g$ is increasing and $g(0) \le 1$.
	Suppose the rank of $Y$ is $r$ and $\tau < \sigma^1_Y$. There exists a unique integer $j$, with $1\le j \le r$, such that
	the solution $t_j$ to the following equation
	\begin{equation}\label{algebraicEquation_g}
		g(\sum_{i=1}^{j}{\sigma^i_Y} - jt_j) = \frac{t_j}{\tau}
	\end{equation}
	satisfies the constraint
	\begin{equation}\label{constraint_g}
		\sigma^{j+1}_Y \le t_j < \sigma^j_Y.
	\end{equation}
\end{lem}

\begin{proof}
	We first show that if at least one such $j$ exists, then such a $j$ (and hence $t_j$) is unique.
	Consider the set
	$J = \{j \mid t_j \text{ satisfies \eqref{algebraicEquation_g} and \eqref{constraint_g}}\}.$
	Assume $J \neq \emptyset$, let $j^*$ be the smallest element in $J$. Now we argue that no $j^*+k$, $1\le k \le r-j^*$, can be in $J$.
	Consider any $k$ with $1\le k \le r-j^*$. Suppose for contradiction $j^*+k \in J$. That is:
	
	\begin{align}
	\label{eq1_g}
	g(\sum_{i=1}^{j^*+k}{\sigma^i_Y} - (j^*+k)t_{j^*+k}) = \frac{t_{j^*+k}}{\tau},\\
	\label{eq2_g}
	\sigma^{j^*+k+1}_Y \le t_{j^*+k} < \sigma^{j^*+k}_Y.
	\end{align}
	
	Expanding on the right side of (\ref{eq1_g}), we have
	\begin{align*}
		&g(\sum_{i=1}^{j^*+k}{\sigma^i_Y} - (j^*+k)t_{j^*+k}) 
		\ge
		g(\sum_{i=1}^{j^*}{\sigma^i_Y} + k\sigma^{j^*+k}_Y - (j^*+k)t_{j^*+k}) \\
		&=  g(\sum_{i=1}^{j^*}{\sigma^i_Y} - j^*t_{j^*+k} + k(\sigma^{j^*+k}_Y - t_{j^*+k}))
	> g(\sum_{i=1}^{j^*}{\sigma^i_Y} - j^*t_{j^*+k})\\
		&> g(\sum_{i=1}^{j^*}{\sigma^i_Y} - j^*t_{j^*}) = \frac{t_{j^*}}{\tau},
	\end{align*}
	where the first inequality follows from the non-increasing values of the
	singular values, the second inequality follows from the assumption in (\ref{eq2_g}) and that $g$ is increasing, the last inequality follows from $t_{j^*} \ge \sigma^{j^*+1}_Y \ge \sigma^{j^*+k}_Y > t_{j^*+k}$ and the last equality follows from the definition of $t_{j^*}$.
	
	Hence, it follows that 
	$$\frac{t_{j^*+k}}{\tau} = g(\sum_{i=1}^{j^*+k}{\sigma^i_Y} - (j^*+k)t_{j^*+k}) > \frac{t_{j^*}}{\tau}$$, leading to
	$t_{j^*+k} > t_{j^*}$, hence a contradiction.

	Next, we prove that $J$ is indeed not empty.
	
	First, we note that by the property of $g$, a unique solution $t_j$ ($> 0$) exists for
	$g(\sum_{i=1}^{j}{\sigma^i_Y} - jt_j) = \frac{t_j}{\tau}$, for each $j$ satisfying $1\le j \le r$.
	We denote by $t_j$ the unique solution corresponding to each $j$. Hence, it suffices to show at least one $t_j$
	satisfies $\sigma^{j+1}_Y \le t_j < \sigma^j_Y$.
	
	Again by monotonicity of $g$ and $g(0) \le 1$, it is easily seen that
	$\tau < \sigma^1_Y$ implies that $t_1 < \sigma^{1}_Y$.
	Now suppose it also holds that $\sigma^{2}_Y \le t_1$, then we are done.
	Otherwise, we have $t_1 < \sigma^{2}_Y$.
	Under this assumption, we claim that $t_1 < t_2$ and $t_2 < \sigma^{2}_Y$.
	To prove $t_1 < t_2$, assume for the sake of contradiction that $t_1 \ge t_2$, leading to:
	\begin{eqnarray*}
		&& \frac{t_1}{\tau} =g(\sigma^1_Y - t_1) < g(\sigma^1_Y + \sigma^2_Y - 2t_1)\\
		&\le& g(\sigma^1_Y + \sigma^2_Y - 2t_2) = \frac{t_2}{\tau},
	\end{eqnarray*}
	where the first inequality follows from $t_1 < \sigma^2_Y$.
	Hence we reach a contradiction, establishing that $t_1 < t_2$. 
	
	The desired inequality $t_2 < \sigma^{2}_Y$ then follows since 
	\begin{eqnarray*}
		&& g(\sigma^1_Y - t_1) = \frac{t_1}{\tau} < \frac{t_2}{\tau} < g(\sigma^1_Y + \sigma^2_Y - t_1 - t_2)\\
	\end{eqnarray*}
	implies $\sigma^1_Y - t_1 < \sigma^1_Y + \sigma^2_Y - t_1 - t_2$
	by monotonicity of $g$, hence yielding $t_2 < \sigma^2_Y$

	If $t_2 \ge \sigma^{3}_Y$, then the claim is established.
	If not, we can repeat this process inductively.
	More formally, suppose we have just finished the $j$-th iteration (note that the induction basis $j = 1$ is verified above) and we have $t_j < \sigma^{j}_Y$.
	If it also holds that $t_j \ge \sigma^{j+1}_Y$, then the claim follows. If not, then we show $t_{j+1} > t_j$ and $t_{j+1} < \sigma^{j+1}_Y$.
	First, assume on the contrary, $t_{j+1} \le t_j$
	\begin{eqnarray*}
		&& \frac{t_j}{\tau} = g(\sum_{i=1}^{j}{\sigma^i_Y} - jt_j) < g(\sum_{i=1}^{j+1}{\sigma^i_Y} - (j+1)t_{j+1}) = \frac{t_{j+1}}{\tau}\\
	\end{eqnarray*}
	where the first inequality follows from $t_j < \sigma^j_Y$.
	Hence we reach a contradiction, establishing that $t_j < t_{j+1}$.
	
	Next, we note that $t_{j+1} < \sigma^{j+1}_Y$ follows since
	\begin{eqnarray*}
		&& g(\sum_{i=1}^{j}{\sigma^i_Y} - jt_j) = \frac{t_j}{\tau}  <  \frac{t_{j+1}}{\tau} < g(\sum_{i=1}^{j+1}{\sigma^i_Y} - jt_{j} - t_{j+1}), \\
	\end{eqnarray*}
	(where the last inequality follows due to $jt_j < jt_{j+1}$,) implying $\sum_{i=1}^{j}{\sigma^i_Y} - jt_j < \sum_{i=1}^{j+1}{\sigma^i_Y} - jt_{j} - t_{j+1}$, which is equivalent to $t_{j+1} < \sigma^{j+1}_Y$.

	Thus, we have a strictly increasing sequence $\{t_j\}$ with $t_j < \sigma^{i}_Y$.
	If it holds that $\sigma^{j+1}_Y \le t_j < \sigma^{j}_Y$ at some iteration $j$, then such a $j$ certifies that $J$ is not empty.
	If $\sigma^{j+1}_Y \le t_j < \sigma^{j}_Y$ never holds for $j$ up to $r-1$, then it must hold for $j = r$, since $-\infty = \sigma^{r+1}_Y \le t_r < \sigma^{r}_Y$, also certifying that $J$ is not empty.
\end{proof}

\begin{rem}
	In addition to asserting the unique existence of such a $j^*$, the proof suggests a natural binary search algorithm to find such a $j^*$ and the corresponding $t_{j^*}$.
	The algorithm is given in Algorithm~\ref{alg:thresholdComputation}.
	Note that the step ``Compute $t_{j}$" can be easily done very efficiently by numerically solving $g(\sum_{i=1}^{j}{\sigma^i_Y} - jt_j) = \frac{t_j}{\tau}$, even though there may not be any analytical solution. 
\end{rem}

\begin{lem}\label{comp}
	Algorithm~\ref{alg:thresholdComputation} correctly computes the unique $j$ and $t_j$ guaranteed by Lemma~\ref{existence}. 
\end{lem}

\begin{proof}
	From the first part of proof for Lemma \ref{existence}, we know that if $j^*$ is the unique $j$ guaranteed by Lemma~\ref{existence}, then for all $k > j^*$, we have $t_k \ge \sigma^{k}_Y$.
	Thus, if $t_M < \sigma^{M}_Y$, then we know that $j^*$ cannot be less than M. That is, $j^*$ must be in the second half of the unsearched space.
	Conversely, if we hypothetically do a sequential search, then it follows immediately from the second part of proof of Lemma~\ref{existence} that before $j$ reaches $j^*$,
	$t_M < \sigma^{M}_Y$ must hold. This establishes that if in the while loop we encounter $t_M \ge \sigma^{M}_Y$, then it must be the case that $j^* \le M$. That is, $j^*$ must lie in the first part of the unsearched space.
	It then follows that $j^*$ always lies between $L$ and $R$, establishing that while loop will eventually halt, returning $t_{j^*}$ and $j^*$.
	
\end{proof}

\begin{algorithm}[tb]
	\caption{Generalized Singular Value Threshold Computation}
	\label{alg:thresholdComputation}
	\begin{algorithmic}
		\STATE {\bfseries Input:} $Y$, $\tau$, $\{\sigma^{i}_Y\}_{1\le i \le r}$
		\STATE {\bfseries Output:} $j^*$ and $t_{j^*}$
		\newline
		
		\STATE {\textbf{if} $\tau \ge \sigma^{1}_Y$, Return $0$ and $\sigma^{1}_Y$}

		\STATE Initialize $L = 1, R = r$,
		\WHILE{$L \le R$}
		\STATE $M = \lceil{\frac{L+R}{2}}\rceil$, and compute the solution $t_{M}$ to the equation $g(\sum_{i=1}^{M}{\sigma^i_Y} - Mt_M) = \frac{t_M}{\tau}$ 
		\IF{$t_M < \sigma^{M}_Y$}
		\STATE{\textbf{if}$\sigma^{M+1}_Y \le t_{M}$, Return  $M$ and $t_{M}$}
		\STATE{\textbf{else} $L = M$,Continue}
		\ENDIF
		\STATE {\textbf{Else} $R = M$, Continue}	
		\ENDWHILE
	\end{algorithmic}
\end{algorithm}

\vspace{-5ex}
\begin{defn}
	Given $\tau >0$, the generalized singular value thresholding operator $\mathcal{H}_\tau$ is define to be:
	$$\mathcal{H}_\tau(Y) = U\mathcal{D}_{t_{j^*}}(\Sigma)V^*, Y = U\Sigma V^* \in \mathbf{R}^{n_1 \times n_2},$$ where $t_{j^*}$ is the threshold computed by Algorithm~\ref{alg:thresholdComputation}.
\end{defn}

Lemma~\ref{existence} guarantees that $\mathcal{H}_\tau$ is well-defined and Algorithm~\ref{alg:thresholdComputation} guarantees that $\mathcal{H}_\tau$ is efficiently computable. Having defined $\mathcal{H}_\tau$, the main result is:

\begin{thm}\label{main}
	Let $f : \mathbf{R} \rightarrow \mathbf{R}$ be any convex, increasing, differentible function, with an increasing derivative satisfying 
	$f'(0) \le 1$. 
	Given a $\tau > 0$ and a $Y \in \mathbf{R}^{n_1 \times n_2}$, we have:
	$$\mathcal{H}_\tau(Y) = \arg\min_{X}\{\tau f({\|X\|_*}) + \frac{1}{2} \|X - Y\|^2_F\}.$$
\end{thm}

\begin{proof}
	To prove this theorem, we build on the techniques introduced in \cite{cai2010singular}.
	
	The function $h(X) = \tau f(\|X\|_*) + \frac{1}{2}\|X-Y\|^2_F$ is strictly convex, since it is the sum of a convex function and a strictly convex function.
	As a result, the minimizer $\hat{X}$ to $h(X)$ is unique and it suffices to show that $\mathcal{H}_\tau(Y)$ is one minimizer.
	
	Per the definition of a subgradient, $S$ is a subgradient of a convex function $f$ at $X_0$ if
	$f(X) \ge f(X_0) + \langle S, X - X_0 \rangle$.
	$\partial f(X_0)$ is commonly used to denote the set of subgradients of $f$ at $X_0$.
	Recall that the set $\partial \|X_0\|_*$ of subgradients of the nuclear norm function at $X_0$ is:
	$\partial \|X_0\|_* = \{U_{X_0}V^*_{X_0} + W \mid W\in \mathbf{R}^{n_1 \times n_2}, 
	U^*_{X_0}W = 0, WV_{X_0} = 0, \|W\|_2 \le 1\},$ where the SVD of $X_0$ is $U_{X_0}\Sigma_{X_0}V^*_{X_0}$ and $\|W\|_2$ is the top singular value of $W$.
	
	First, it is easy to check that
	$f'({\|X_0\|_*})(U_{X_0}V^*_{X_0} + W) \in \partial (f({\|X_0\|_*}))$ for $W\in \mathbf{R}^{n_1 \times n_2}, U^*_{X_0}W = 0, WV_{X_0} = 0, \|W\|_2 \le 1$, by the composition rule for the subgradient. 
	Hence, $\tau g({\|\hat{X}\|_*})(U_{\hat{X}}V^*_{\hat{X}} + W) + \hat{X} - Y$ is a subgradient for $h$ at $\hat{X}$,
	for $W$ satisfying $U^*_{\hat{X}}W = 0, WV_{\hat{X}} = 0, \|W\|_2 \le 1$, where $g = f'$ and $g$ satisfies the assumption given in Lemma~\ref{existence}.
	Moreover, if there exists such a $W$ and it holds that $0 = \tau g({\|\hat{X}\|_*})(U_{\hat{X}}V^*_{\hat{X}} + W) + \hat{X} - Y$, or
	equivalently that
	\begin{equation}\label{optimality}
	Y- \hat{X} = \tau g({\|\hat{X}\|_*})(U_{\hat{X}}V^*_{\hat{X}} + W),
	\end{equation}
	then $\hat{X}$ is a minimizer (hence the unique minimizer) to $h(X)$.
	
	We now establish that, with $\hat{X} = \mathcal{H}_\tau(Y)$, Eq.~(\ref{optimality}) does hold with $W$ satisfying the given constraints.
	First, we consider the case that $\tau < \sigma^1_Y$.
	
	By Lemma~\ref{existence}, since $t_{j^*}$ satisfies the equation 
	$\frac{t_{j^*}}{\tau} = g({\sum_{i=1}^{j^*}\sigma^i_Y} - j^*t_{j^*})$, we have
	$t_{j^*} = \tau g({\sum_{i=1}^{j^*}{(\sigma^i_Y} - t_{j^*})})$.
	Since $\sigma^{j^*+1}_Y \le t_{j^*} < \sigma^{j^*}_Y$, $Y$'s last $r-j^*$ singular values ($\sigma^{k}_Y$ with $k \ge j^*+1$) have been set to 0, leading to that $\|\hat{X}\|_* = \sum_{i=1}^{j^*}{(\sigma^i_Y - t_{j^*})}$.
	Therefore, $t_{j^*} = \tau g({\|\hat{X}\|_*})$.
	
	Next, we partition $Y$ as follows:
	$$Y = U_a \Sigma_a V_a^* + U_b \Sigma_b V_b^*,$$
	where $U_a$'s columns and $V_a$'s columns are the first $j^*$ left and right singular vectors respectively, associated with the first $j^*$ singular values (i.e. singular values larger than $t_{j^*}$), while
	$U_b$'s columns and $V_b$'s colunms are the remaining $r - j^*$ left and right singular vectors respectively, associated with the remaining $r - j^*$ singular values (i.e. singular values less than or equal to $t_{j^*}$).
	
	Under this partition, it is easily seen that
	$\hat{X} =  U_a (\Sigma_a - t_{j^*} \mathbf{I}) V_a^*.$
	We then have
	\begin{eqnarray*}
		&&Y - \hat{X} = t_{j^*}U_aV_a^* + U_b \Sigma_b V_b^*  \\
		&=& \tau g({\|\hat{X}\|_*})U_aV_a^* + U_b \Sigma_b V_b^*  \\
		&=& \tau g({\|\hat{X}\|_*})(U_aV_a^* + \frac{1}{t_{j^*}}U_b \Sigma_b V_b^*).
	\end{eqnarray*}
	Choose $W = \frac{1}{t_{j^*}}U_b \Sigma_b V_b^*$. By construction,
	$U_a^*U_b = 0, V_b^*V_a =0$, hence we have $U^*_{\hat{X}}W = 0, WV_{\hat{X}} = 0$.
	In addition, since $t_{j^*} \ge \sigma^{j*+1}_Y$, we have $\|W\|_2 = \frac{\sigma^{j*+1}_Y}{t_{j^*}} \le 1$.
	Hence $W$ thus chosen satisfies the constraints, hence establishing the claim.
	
	Now, if $\tau \ge \sigma^1_Y$, then $t_{j^*} = \sigma^1_Y$ and $j^* = 1$ are returned by Algorithm~\ref{alg:thresholdComputation}.
	It follows immediately that in this case $\hat{X} = 0$. Choosing $W = \tau^{-1}Y$,
	it is easily seen that $W$ satisfies the constraints. Verification of Eq.~\eqref{optimality} is instant when $\hat{X} = 0$ and $W = \tau^{-1}Y$.
\end{proof}


\bibliographystyle{apalike}
\bibliography{references_note}

\end{document}